\documentclass[12pt]{amsart}
\usepackage{amsfonts,latexsym,rawfonts,amsmath,amssymb,amsthm,graphicx}
\newcommand{\mint}{-\hspace{-0.9em}\int}
\newcommand{\mintdis}{-\hspace{-1.1em}\int}
\numberwithin{equation}{section}
\newtheorem{theorem}{Theorem}[section]
\newtheorem{lem}[theorem]{Lemma}
\newtheorem{thm}[theorem]{Theorem}
\newtheorem{pro}[theorem]{Proposition}

\newcounter{Cnumber}

\def\s{\,\,\,\,}
\def\lan{\langle}
\def\ran{\rangle}

\def\dint{\displaystyle{\int}}
\def\mv{1.7ex}
\def\endproof{$\hfill\Box$\\}
\def\R{\mathbb{R}}

\def\C{\mathbb{C}}
\def\V{{\mathcal V}}

\def\W{{\mathcal W}}
\def\C{\mathbb{C}}
\def\S{\mathbb{S}}

\def\N{\mathbb{N}}
\def\ve{\varepsilon}

\def\div{{\rm div\,}}

\title[Willmore minimizers with small isoperimetric ratio]
{\bf Asymptotics of Willmore minimizers with 
prescribed small isoperimetric ratio}
\author{{\sc Ernst Kuwert} and {\sc Yuxiang Li}}

\address{\newline
Ernst Kuwert:
Mathematisches Institut,
Albert-Ludwigs-Universit\"at Freiburg,
D-79104 Freiburg, Germany
\textit{E-mail: ernst.kuwert@math.uni-freiburg.de}
\newline
\newline
Yuxiang Li:
Department of Mathematical Sciences,
Tsinghua University, Beijing 100084, P.R. China, 
\textit{E-mail: yxli@math.tsinghua.edu.cn}}

\date{}
\begin{document}
\maketitle
\begin{abstract} 
{\noindent 
We consider surfaces in ${\mathbb R}^3$ of type ${\mathbb S}^2$ 
which minimize the Willmore functional with prescribed 
isoperimetric ratio. In \cite{Schy12} Schygulla proved 
the existence of smooth minimizers. In the singular limit 
when the isoperimetric ratio converges to zero, he showed 
convergence to a double round sphere in the sense of 
varifolds. Here we give a full blowup analysis of 
this limit, showing that the two spheres are connected 
by a catenoidal neck. Besides its geometric interest, 
the problem was studied as a simplified model in the 
theory of cell membranes, see e.g. \cite{BLS91}.}
\end{abstract}

\section{Introduction}

The isoperimetric ratio of a smooth embedding $f:{\mathbb S}^2 \to \R^3$
is defined as 
\begin{equation}
\label{defratio}
\sigma(f) = 6 \sqrt{\pi}\, \frac{\V(f)}{\mu(f)^{3/2}} \in (0,1].
\end{equation}
The area $\mu(f)$ and enclosed volume $\V(f)$ are given by
\begin{equation}
\mu(f) = \int_{{\mathbb S}^2} J\!f\,d\mu_{\S^2} \quad \mbox{ and } \quad
{\mathcal V}(f) 
= \frac{1}{3} \int_{{\mathbb S}^2} \langle f,\vec{n} \rangle\,d\mu_f,
\end{equation}
where $\vec{n}:\S^2 \to \S^2$ is the exterior unit normal. The
normalization is chosen such that $\sigma(\S^2) = 1$. It is an 
interesting geometric problem to find natural representatives among 
surfaces with prescribed  isoperimetric ratio $\sigma \in (0,1)$.
In \cite{Schy12} Schygulla introduced a variational approach 
by minimizing the Willmore energy
$$
\W(f) = \frac{1}{4} \int_{\S^2} |\vec{H}|^2\,d\mu_f.
$$
He proved that the infimum among topological spheres
is attained by a smooth minimizer, for any prescribed 
isoperimetric ratio $\sigma \in (0,1]$. An existence result 
for higher genus surfaces was proved more recently by Keller,
Mondino and Rivi\`{e}re, assuming that the infimum of 
the energy satisfies certain inequalities \cite{KMR14}. 
In any case minimizers solve the Euler Lagrange equation
\begin{equation}
\label{eqeuler}
\frac{1}{2}\big(\Delta_g H + |A^\circ|^2 H\big) = 
\Lambda\, \sigma(f)\,\Big(\frac{1}{\V(f)} + \frac{3}{2\mu(f)} H\Big). 
\end{equation}
Here the left hand side is the Euler Lagrange operator
of the Willmore energy, and the right hand side is the 
Euler Lagrange operator of the isoperimetric ratio. 
By Alexandrov's theorem the isoperimetric constraint is 
nondegenerate and hence the Lagrange multipier $\Lambda \in \R$ 
is well-defined, with the only exception of round spheres.\\ 
\\ 
In his paper \cite{Schy12} Schygulla also studies 
sequences of minimizers with isoperimetric ratio 
converging to zero. He shows that up to translations and dilations, 
the sequence converges in the varifold sense to a round sphere of 
multiplicity two \cite[Thm. 2]{Schy12}. In the present paper
we study this singular limit more precisely, and obtain the 
following asymptotic results. Here we denote by $N,S$ the
north and south pole of $\S^2$ and by $\pi_N,\pi_S$ the 
stereographic projections from the poles. 

\begin{thm}\label{main}
Let $f_k:\S^2\rightarrow\R^3$ be conformally parametrized $\W$-minimizers
for prescribed isoperimetric ratio $\sigma(f_k) = \sigma_k \to 0$.
After conformal reparametrization, scaling and translating, 
and passing to a subsequence, the following hold:

\begin{itemize}
\item[{\rm (1)}] $\mu(f_k) = 1$, \vspace{2mm} 
\item[{\rm (2)}] $f_k$ converges locally smoothly on $\S^2 \backslash \{N\}$
to a conformal immersion $f^0:\S^2 \to \R^3$ with $f^0(N)=0$. \vspace{2mm}
\item[{\rm (3)}] There exist $r_k \to 0$  such that the sequence
$f^1_k(z) = f_k \circ \pi_S^{-1}(r_k z)$ converges
locally smoothly on $\R^2$ to a conformal immersion
$f^1:\R^2 \to \R^3$ with $f^1(\infty)=0$.\vspace{2mm}
\item[{\rm (4)}] Both $f^0,f^1$ are conformal equivalences
to the same round sphere $\S$ of area $1/2$. \vspace{2mm}
\item[{\rm (5)}] There exist $t_k,\,\lambda_k \rightarrow 0$, 
such that the sequence $f_k^2(z) = \lambda_k^{-1} f_k \circ \pi_S^{-1}(t_k z)$
converges locally smoothly on $\R^2 \backslash \{0\}$ to a conformally
parametrized catenoid $f^2:\R^2 \backslash \{0\} \to \R^3$, with 
the origin and the center of $\S$ on the symmetry axis. \vspace{2mm}
\item[{\rm (6)}] We have the energy identity
$\lim_{k\rightarrow+\infty}\int_{\S^2}|A_{f_k}|^2\,d\mu_{f_k} 
= \sum_{i = 0,1,2} \int |A_{f^i}|^2\,d\mu_{f^i}$. \vspace{2mm}
\item[{\rm (7)}] We have 
$\lim_{k\rightarrow \infty}\big(\log t_k:\log\lambda_k:\log r_k\big)=
(1:1:2)$, and\, $\lim_{k\rightarrow \infty}\Lambda_k/\lambda_k$ exists.
\end{itemize}
\end{thm}

The geometric problem we adress here is partially motivated
by a model for cell membranes due to Helfrich \cite{Helf73}. 
In the case of topological spheres the Helfrich energy becomes
$$
F_{SC} = \frac{\varkappa}{2} \int_{\S^2} (H-C_0)^2\,d\mu_g
+ 4\pi \varkappa_G,
$$
where $\varkappa,\varkappa_G$ and $C_0$ are constants. 
The Helfrich energy essentially reduces to the Willmore 
energy when the spontaneous curvature $C_0$ is zero. 
Although this simplified case may not correspond to real 
membranes, it was studied by several authors. In the 
axially symmetric case the minimizers are described
by ordinary differential equations which were studied by 
numerical approximations, see e.g. \cite{Can70,Lip91,BLS91}.
In particular Berndl, Lipowsky and Seifert \cite{BLS91} 
distinguished three kinds of shapes depending on the 
isoperimetric ratio $\sigma$, which they call the reduced 
volume: the prolate-dumbbell, the oblate-discocyte
and the stomatocyte type. The stomatocyte parameter
range is indicated as $0 < \sigma \leq 0.591$, and it 
is noted that a neck develops as $\sigma \to 0$. Here
we provide a rigorous analysis for this neck formation,
without assuming axial symmetry.

\section{Bubble Limits of Minimizers $f_k$ for Isoperimetric Ratio $\sigma_k \to 0$}
Let $\Sigma$ be a Riemann surface. We denote by $W^{2,2}_{conf}(\Sigma,\R^3)$
the set of $f \in W^{2,2}_{loc}(\Sigma,\R^3)$, such that in any conformal
chart the metric has the form $g_{ij} = e^{2u} \delta_{ij}$
where $u \in L^\infty_{loc}$ \cite{KL12}. More generally,
a branched conformal immersion on $\Sigma$ is a map 
$f \in W^{2,2}_{conf}(\Sigma\backslash {\mathcal S},\R^3)$ 
where ${\mathcal S}$ is finite and 
$$
\mu_g(\Omega \backslash {\mathcal S}) 
+ \int_{\Omega \backslash {\mathcal S}}|A|^2\,d\mu_g < \infty
\quad \mbox{ for any } \Omega \subset\!\! \subset \Sigma.
$$
Any branched conformal immersion is in $W^{2,2}_{loc}(\Sigma,\R^3)$.
At each $p \in \Sigma$ the map has multiplicity $m_p \in \N_0$ 
and behaves like $z^{m_p}$ in a weak sense.\\ 
\\
In this section we summarize results from \cite{CL14} about the 
convergence of a sequence $f_k:\S^2 \to \R^3$ of (branched) 
conformal immersions with bounded Willmore energy and area. We 
first recall the notion of a bubble: a sequence $z_k \in \C$, 
$r_k > 0$, with $z_k \to 0$ and $r_k \to 0$ is called a 
bubble of $f_k$ at $p \in \S^2$, if the following holds: 
in some local conformal coordinates with $p$ corresponding  
to $z = 0$ the sequence $\hat{f}_k(z) = f_k(z_k + r_k z)$
converges locally $W^{2,2}$-weakly on $\C$, away from 
finitely many points, to a nonconstant map $f:\C \to \R^3$.
It follows that when composed with stereographic projection, 
the blowup limit $f$ is a branched conformal immersion from 
$\S^2$ to $\R^3$ with finite area and finite Willmore energy. 
This notion of a bubble is clearly independent of the 
coordinates. Two bubbles $(z_k^i,r_k^i)$, $i =  1,2$, are 
different if they concentrate at different points 
$p_i \in \S^2$, or if their scales are different in the 
sense that
$$
\frac{r_k^1}{r_k^2}+\frac{r_k^2}{r_k^1}
+\frac{|z_k^1-z_k^2|}{r_k^1+r_k^2} \rightarrow \infty.
$$
To construct a limit of the sequence $f_k$ we first observe 
the diameter of $f_k(\S^2)$ is bounded. In the case 
when $f_k$ is a smooth immersion, the diameter estimate
was proved by Simon \cite{S} using a monotonicity formula.
For a $W^{2,2}$-conformal immersion one can approximate
by smooth immersions to get the result. If the $f_k$ are 
merely branched conformal immersions one considers
the image $2$-varifold as in \cite{KS04}. Now
after translating we can assume 
$$
f_k(\S^2) \subset \overline{B_R(0)} \quad \mbox{ for all } k.
$$
Together with the Willmore energy bound, this shows that $f_k$ 
belongs to the class ${\mathcal F}^2(\S^2,g_{\S^2},R)$ defined 
in \cite[Sec.2]{CL14}. Passing to a subsequence, we can also 
arrange that 
$$
\|A_{f_k}\|_{g_k}^2\,d\mu_{g_k} \to \alpha \quad \mbox{ in } C^0(\S^2)'.
$$
Depending on a constant $\ve_0 > 0$, which will be chosen
appropriately, the finite concentration set is defined by 
$$
{\mathcal S} = \{p \in \S^2:\alpha(\{p\}) \geq \ve_0\}.
$$
Proposition 2.1 in \cite{CL14} says that there is a subsequence
such that $f_k \to f^0$ weakly in $W^{2,p}$ for any $p < 2$ and 
strongly in $W^{1,q}$ for any $q < \infty$, away from ${\mathcal S}$.
In particular we can assume $Df_k \to Df^0$ pointwise 
almost everywhere. By H\'{e}lein's convergence result, 
see e.g. \cite[Thm. 5.1]{KL12}, we get for a subsequence 
the following alternative: 
\begin{itemize}
\item[{\rm (a)}] either $u_k$ is locally bounded in 
$\S^2\backslash \mathcal{S}$, and $f_k$ converges $W^{2,2}$-weakly 
away from $\mathcal{S}$ to a branched conformal immersion
$f:\S^2 \to \R^3$. 

\item[{\rm (a)}] or $u_k \to -\infty$ and $f_k \to {\rm const.}$ 
locally uniformly on $\S^2\backslash \mathcal{S}$.
\end{itemize}

To rule out the case that the limit $f^0$ is constant, we will 
apply a general bubbling result from \cite[Thm. 2.8]{CL14}.
In the case of surfaces of type $\S^2$ it says the 
following.

\begin{lem} \label{chenli} Let $f_k:\S^2 \to \R^3$  be a sequence of (branched)
conformal immersions with bounded Willmore energy and area.
Then there exists a branched conformal immersion $f^0:\S^2 \to \R^3$
(possibly constant) and a set of different bubbles $(z^i_k,r^i_k)$, $1 \leq i \leq N$
(possibly empty), such that the following holds for a subsequence:
\begin{eqnarray}
\label{eqchenli1}
\lim_{k \to \infty} \int_{\S^2} d\mu_{f_k} & = & 
\sum_{i=0}^N \int_{\S^2} d\mu_{f^i},\\
\label{eqchenli2}
\limsup_{k \to \infty} \W(f_k) & \leq & \sum_{i=0}^N \W(f^i).
\end{eqnarray}
\end{lem}


\begin{lem}\label{lemmatwospheres} 
Let $f_k:\S^2 \to \R^3$ be minimizers for the Willmore 
energy with prescribed isoperimetric ratio $\sigma_k = \sigma(f_k) \to 0$.
After composing with suitable M\"obius transformations, the following holds:
\begin{itemize}
\item[{\rm (a)}] $f_k$ converges to a conformal equivalence $f^0:\S^2 \to \S$,
where $\S$ is a round sphere of area $1/2$ passing through the origin. 
\item[{\rm (b)}] $f_k$ has a bubble $f^1:\C \to \S$ which extends to a 
                 conformal equivalence $f^1:\hat{\C} \to \S$ with orientation
                 opposite to $f^0$.
\end{itemize}
Moreover, $f_k$ has no further bubbles. 
\end{lem} 

\proof In \cite{Schy12} Schygulla proved that up to translations
$f_k(\S^2) \to  2\,\S$ as varifolds, where $\S$ is as stated. 
The monotonicity formula
in \cite{S} implies that the convergence is also in Hausdorff distance.
We claim that when applying Lemma \ref{chenli}, the map $f^0$ 
can be assumed nonconstant. 
Otherwise by (\ref{eqchenli2}) there is at least one bubble 
$f^1:\S^2 \to \R^3$, having area $\mu(f^1) > 0$. Up to rotation,
the bubble concentrates at the north pole, which means 
$$
f_k(\pi_S^{-1}(z_k + r_k z)) \to f^1(z) \quad 
\mbox{ $W^{2,2}$-weakly locally in $\R^2$, away from a finite set.}
$$
The maps $\pi_S^{-1} \circ A_k \circ \pi_S$, $A_k(z) = z_k + r_k z$,
extend to M\"obius transformations of $\S^2$, and 
$$
f_k(\pi_S^{-1} \circ A_k \circ \pi_S) \to f^1 \circ \pi_S \neq {\rm const.}
$$
Thus we obtained a nonconstant limit $f^0: = f^1 \circ \pi_S:\S^2 \to \S$, 
proving the claim. Now $f^0$ is a branched conformal immersion, in 
particular $f^0 \in W^{2,2} \cap W^{1,\infty}(\S^2,\R^3)$. We compute 
in local complex coordinates on $\S^2$, writing
$\langle \partial_i f^0,\partial_j f^0 \rangle = e^{2u_0} \delta_{ij}$,
$$
\langle \Delta f^0,\partial_j f^0 \rangle = 
\partial_i \langle \partial_i f, \partial_j f \rangle 
- \frac{1}{2} \partial_j |\partial_i f|^2
= \partial_i (e^{2u_0} \delta_{ij}) - \frac{1}{2} \partial_j (2 e^{2u_0})
= 0.
$$
So $f^0:\S^2 \to \S$ is weakly harmonic, and hence
smooth by standard regularity theory \cite[Thm. 2.5.1]{Jos91}. 
Orienting $\S^2$ and $\S$ by their exterior normals, we conclude 
that $f^0:\S^2 \to \S$ is either holomorphic or anti-holomorphic.
This implies $\mu(f^0) = d\, |\S|$ for $d \in \{1,2\}$.\\
\\
Assume by contradiction $d = 2$. Then $\mu(f^0) = 1 = \mu(f_k)$, which 
implies $f_k \to f^0$ in $W^{1,2}(\S^2,\R^3)$ by conformality. But the 
volume functional is continuous under $W^{1,2}$-convergence 
of uniformly bounded maps, and so 
$\V(f^0) = \lim_{k \to \infty}  \V(f_k) = 0$, a contradiction.
Thus we have $d = 1$ and $f^0:\S^2 \to \S$ is a conformal 
automorphism. Now by (\ref{eqchenli1}) there exists a
bubble $f^1:\S^2 \to \S$. The argument above shows that this 
is again a conformal equivalence. Now for any collection
of different bubbles we have the lower semicontinuity
$$
\mu(f^0) + \sum_{i=1}^N \mu(f^i) \leq \liminf_{k \to \infty} \mu(f_k) = 1.
$$
Since $\mu(f^0) + \mu(f^1) = 1$, a further bubble would have zero 
area and again be conformal, hence constant. This is ruled out
by the definition of bubble.\\
\\
To show that $f^{0}$ and $f^1$ are oppositely oriented, we compute 
with $\omega = \frac{1}{3} x \llcorner dx^1 \wedge dx^2 \wedge dx^3$ 
$$
\int_{\S^2 \backslash B_\varrho(N)} f_k^\ast \omega
+ \int_{D_{\frac{1}{\varrho}(0)}} \hat{f}_k^\ast \omega
= \int_{\S^2 \backslash A_{k,\varrho}} f_k^\ast \omega 
= \V(f_k) - \int_{A_{k,\varrho}} f_k^\ast \omega,
$$
where 
$A(k,\varrho) =  \S^2 \backslash 
\big(B_\varrho(N) \cup \pi_S^{-1}(D_{\frac{r_k}{\varrho}}(z_k)\big)\big)$.
The error term is estimated by 
$$
\Big|\int_{A_{k,\varrho}} f_k^\ast \omega\Big| \leq 
C \int_{A_{k,\varrho}} d\mu_{f_k}
= C \,\Big(1- \int_{\S^2 \backslash B_\varrho(N)} d\mu_{f^k}
- \int_{D_{\frac{1}{\varrho}}(0)} d\mu_{\hat{f}_k}\Big).
$$
Letting $k \to \infty$ we obtain using $\lim_{k \to \infty} \V(f_k) = 0$
\begin{eqnarray*}
\Big| \int_{\S^2 \backslash B_\varrho(N)} (f^0)^\ast \omega 
+ \int_{D_{\frac{1}{\varrho}}(0)} (f^1)^\ast \omega \Big| & \leq &
\limsup_{k \to \infty} \Big|\int_{A_{k,\varrho}} f_k^\ast \omega\Big|\\
& \leq & C\Big(1 - \int_{\S^2 \backslash B_\varrho(N)} d\mu_{f^0}
- \int_{D_{\frac{1}{\varrho}}(0)} d\mu_{f^1}\Big).
\end{eqnarray*}
Letting $\varrho \to 0$ proves our claim.
\endproof

\section{Estimates for Critical Points with Prescribed
Isoperimetric Ratio}

The first variation of the Willmore energy at $f:\Sigma \to \R^3$ 
in direction $\phi$ is given by
\begin{eqnarray}
\nonumber
\delta\W(f)\phi & = &
\frac{1}{2} \int_\Sigma \langle \vec{H},\Delta_g\phi \rangle\,d\mu_g\\
\label{eqEL1} 
&& - \int_\Sigma g^{\alpha \beta} g^{\lambda \mu}
\langle \vec{H},A_{\alpha \lambda} \rangle \langle \partial_\beta f,\partial_\mu\phi \rangle\,d\mu_g\\
\nonumber
&& + \frac{1}{4} \int_\Sigma |\vec{H}|^2 g^{\alpha \beta} \langle \partial_\alpha f,\partial_\beta \phi \rangle\,d\mu_g.
\end{eqnarray}
Writing $\vec{H} = H \vec{n}$ and $\phi = \varphi \vec{n}$, we get by 
partial integration for $\phi$ having compact support
\begin{equation}
\label{eqwillmoreoperator}
\delta\W(f)\phi = \frac{1}{2} \int_\Sigma
(\Delta_g H + |A^\circ|^2 H) \varphi \,d\mu_g.
\end{equation}
The first variation of the isoperimetric ratio is 
\begin{eqnarray}
\nonumber
\delta\sigma(f)\phi & = & \sigma(f)\Big(\frac{1}{\V(f)} \int_\Sigma \langle \phi,\vec{n} \rangle\,d\mu_g 
- \frac{3}{2 \mu(f)} \int_\Sigma \langle df,d\phi \rangle_g\,d\mu_g \Big) \\
\label{eqfirstvariationratio}
& = & \sigma(f)\Big(\frac{1}{\V(f)} \int_\Sigma \varphi \,d\mu_g
+ \frac{3}{2 \mu(f)} \int_\Sigma H \varphi\,d\mu_g\Big). 
\end{eqnarray}
The Euler Lagrange operators, i.e. the $L^2$ gradients, of the functionals are thus
\begin{eqnarray} 
\label{eqwillmoregradient}
\nabla \W(f) & = & \frac{1}{2}\big(\Delta_g \vec{H} + |A^\circ|^2 \vec{H}\big),\\
\label{eqisopgradient}
\nabla \sigma(f) & = & \sigma(f) \Big(\frac{1}{\V(f)} \vec{n} + \frac{3}{2\mu(f)} \vec{H}\Big).
\end{eqnarray}
For $\lambda > 0$ we have the scaling $\nabla \W(\lambda f) = \lambda^{-3} \nabla \W(f)$
and also $\nabla \sigma(\lambda f) = \lambda^{-3} \nabla \sigma(f)$.\\
\\
Note that $\vec{H} = H \vec{n}$ with $\vec{n}$ the exterior 
normal implies $H = -2$ for $\S^2$. 

\begin{lem} Let $f:\S^2 \to \R^3$ be a smoothly embedded
Willmore minimizer with prescribed isoperimetric ratio 
$\sigma(f) \in (0,1)$. Then 
\begin{equation}
\label{eqmultiplier}
\delta \W(f) = \Lambda\, \delta \sigma(f) \quad \mbox{ for some }\Lambda \in \R.
\end{equation}
\end{lem}

\proof Assume that $\delta \sigma(f)$ is identically zero. Then
\begin{equation}
\label{eqdegenerate}
H \equiv - \frac{2\mu(f)}{3 \V(f)}.
\end{equation}
By Alexandrov's or by Hopf's theorem, $f$ parametrizes a round 
sphere which contradicts the assumption $\sigma(f) < 1$. Thus 
$\delta \sigma(f)\phi \neq 0$ for some $\phi$, and 
the claim follows by the Lagrange multiplier rule. 
\endproof

Scaling to $\mu(f) = 1$ yields $\sigma = 6 \sqrt{\pi}\, \V(f)$, and equation 
(\ref{eqmultiplier}) becomes 
\begin{equation} 
\label{EL2}
\frac{1}{2} \big(\Delta_g H + |A^\circ|^2 H\big) = 
\Lambda\,\sigma(f)\Big(\frac{6\sqrt{\pi}}{\sigma(f)} + \frac{3}{2} H\Big)
= \frac{3\Lambda}{2} \Big(4\sqrt{\pi} + \sigma H\Big).
\end{equation}




For $f:D \to \R^3$ conformal the first variation formula as given in 
\cite{Riv08} is
\begin{eqnarray}
\label{eqwillmorevariationconformal}
\delta \W(f)\phi = \int_D \langle Q[f],d\phi \rangle 
& \mbox{ for } &
Q[f] = \frac{1}{2} \big(\nabla\vec{H} - \frac{3}{2}H \,\nabla\vec{n} 
+ \frac{1}{2} \vec{H} \times \nabla^\perp \vec{n}\big),\\
\delta \sigma(f)\phi = \int_D \langle R[f],\phi \rangle
& \mbox{ for } &
R[f] = \frac{\sigma(f)}{\V(f)} \partial_1 f \times \partial_2 f
- \frac{3\sigma(f)}{2\mu(f)} \Delta f. 
\end{eqnarray}
The scaling is $Q[\lambda f] = \lambda^{-1} Q[f]$ and also
$R[\lambda f] = \lambda^{-1} R[f]$. The Euler Lagrange equation 
(\ref{eqmultiplier}), scaled to $\mu(f) = 1$, has the form 
\begin{equation}
\label{eqeulerlagrangeconform}
\div Q[f] = \frac{3\Lambda}{2} 
\big(4 \sqrt{\pi}\, \partial_1 f \times \partial_2 f
- \sigma \Delta f\big).
\end{equation}

By the results in Appendix 1  estimates for $|\nabla^m f|$
follow once the boundedness of the Lagrange multiplier 
$\Lambda$ is established.

\begin{lem}
Let $f_k:D \rightarrow \R^3$ be smooth conformal immersions, 
satisfying for $\sigma_k \to 0$ 
$$
\delta \W(f_k) \phi = \frac{3\Lambda_k}{2} 
\int_D \langle 4\sqrt{\pi}\vec{n} + \sigma_k \vec{H}_{f_k}, \phi \rangle\,d\mu_{f_k} 
\quad \mbox{ for all }\phi \in C^\infty_c(D,\R^3).
$$
Let $g_k = e^{2u_k} \delta$ be the induced metrics, and assume 
\begin{equation}
\label{eqmultiplierassumption}
\|A_{f_k}\|_{L^2(D)} + \|u_k\|_{W^{1,2} \cap L^\infty(D)} \leq C < \infty.
\end{equation}
Then the sequence $\Lambda_k$ is bounded. 
\end{lem}

\proof The first variation of the Willmore energy is estimated using 
(\ref{eqEL1}) and (\ref{eqmultiplierassumption}) by
$$
|\delta \W(f_k,\phi)| \leq C \Big(\|\Delta \phi\|_{L^2(D)} + \|D\phi\|_{L^\infty(D)}\Big).
$$
On the right hand side of the equation, we first note using (\ref{eqmultiplierassumption}) 
$$
\int_D |H_{f_k}|\,d\mu_{f_k} \leq \|H_{f_k}\|_{L^2(g_k)} \,\mu(f_k)^{1/2} \leq C.
$$
By (\ref{eqmultiplierassumption}) $f_k$ is bounded in $W^{2,2}(D,\R^n)$.
After passing to a subsequence, we have $f_k \to f$ strongly in $W^{1,2}(D,\R^n)$
and $\mu(f) = \lim_{k \to \infty} \mu(f_k) \in (0,\infty)$. Now 
$$
\pm \int_D \langle \vec{n}_k,\phi \rangle\,d\mu_{f_k}
= \int_D \langle \partial_1 f_k \times \partial_2 f_k,\phi \rangle\,dx
\to \int_D \langle \partial_1 f \times \partial_2 f,\phi \rangle\,dx.
$$
For suitable $\phi$ the right hand side is nonzero. 
The bound for $\Lambda_k$ follows. 
\endproof

{\bf Remark. }In the case $\sigma_k \to \sigma > 0$ the argument 
is modified as follows. We have 
$$
\int_D \langle \vec{H}_{f_k},\phi \rangle\,d\mu_{f_k} = 
\int_D \langle \Delta f_k, \phi \rangle\,dx  \to 
\int_D \langle \Delta f, \phi \rangle\,dx. 
$$
Thus one gets a bound for $\Lambda_k$ unless $f$ satisfies the equation
$$
-\Delta f  = \frac{4 \sqrt{\pi}}{\sigma} \partial_1 f \times \partial_2 f,
$$
i.e. $f$ is a conformally parametrized with constant mean curvature 
$H_f = \pm 4 \sqrt{\pi}/\sigma$.


\section{Construction of the Catenoid Neck} 


Let $f_k:\S^2 \to \R^3$ be the sequence of minimizers for 
$\sigma(f_k) = \sigma_k \to 0$ from Lemma \ref{lemmatwospheres}.
Thus $f_k \to f^0$ weakly in $W^{2,2}$ away from the north pole, and
$f_k$ has a bubble $f^1$ concentrating at the north pole. The maps 
$f^0,f^1:\S^2 \to \S$ are conformal diffeomorphisms, with opposite 
orientations. In particular 
$$
\W(f^0) = \W(f^1) = 4\pi \quad \mbox{ and } \quad 
\int_{\S^2} |A_{f^0}|^2\,d\mu_{f^0} 
= \int_{\S^2} |A_{f^1}|^2\,d\mu_{f^1} = 8\pi.
$$
On the other hand we know that 
$$
\W(f_k) \to 8\pi \quad \mbox{ and } \quad 
\int_{\S^2} |A_{f_k}|^2\,d\mu_{f_k} = 4 \W(f_k) - 8\pi \to 24\pi.
$$
In the following we study $f_k$ near the north pole using the 
conformal coordinates given by the projection $\pi_S$. 
For convenience we denote $f_k \circ \pi_S^{-1}$ again by $f_k$.
Thus putting $f_k^1(z) = f_k(z_k + r_k z)$, we have $f_k^1 \to f^1(z)$ 
locally smoothly on $\C$ up to a finite set, according to Lemma 
\ref{lemmatwospheres}. By rotating we can arrange that $z_k = 0$, 
which means
$$
f^1_k(z) = f_k(r_kz) \to f^1(z) \quad \mbox{ for all }z \in \C.
$$ 
Now by lower semicontinuity of the Willmore energy, we have 
\begin{eqnarray*}
4\pi & = & \lim_{r \searrow 0} \W(f^0,\S^2 \backslash D_r) 
       \leq \lim_{r \searrow 0} \liminf_{k \to \infty} \W(f_k,\S^2 \backslash D_r),\\
4\pi & = & \lim_{r \searrow 0} \W(f^1,D_{\frac{1}{r}}) 
       \leq \lim_{r \searrow 0} \lim_{k \to \infty} \W(f^1_k,D_{\frac{1}{r}})
       = \lim_{r \searrow 0} \lim_{k \to \infty} \W(f_k,D_{\frac{r_k}{r}}).
\end{eqnarray*}
We conclude that
\begin{eqnarray}
\label{minimal}
\lim_{r\to 0} \lim_{k\to \infty}\W(f_k,D_r\backslash D_\frac{r_k}{r}) & = & 0,
\quad \mbox{ and similarly}\\
\label{no.concentration}
\lim_{r \to 0} \lim_{k \to \infty} \int_{D_r\backslash 
D_\frac{r_k}{r}}|A_{f_k}|^2\,d\mu_{f_k} & \leq & 8\pi.
\end{eqnarray}
To find the neck we fix some number $\delta > 0$ and chose 
$t_k \in [\frac{r_k}{\delta},\delta]$ such that 
$$
{\rm diam\,}f_k(\partial D_{t_k}) 
= \min_{t \in [\frac{r_k}{\delta},\delta]} {\rm diam\,}f_k(\partial D_t)
=: \lambda_k > 0.
$$
Then we have $\lambda_k \to 0$, in fact for any $r \in (0,\delta]$ we have 
$$
\limsup_{k \to \infty} \lambda_k 
\leq \limsup_{k \to \infty} \,{\rm diam\,}f_k(\partial D_r) 
= {\rm diam\,}f^0(\partial D_r) \to 0 \mbox{ for } r \to 0.
$$
This implies that also $t_k \to 0$: if we had $t_k \to r > 0$
for a subsequence, then we would get
$$
\lambda_k = {\rm diam\,}f_k(\partial D_{t_k}) \to {\rm diam\,}f^0(\partial D_r) > 0.
$$
Similarly $t_k/r_k \to \infty$, because if $t_k/r_k \to R < \infty$
for a subsequence then 
$$
\lambda_k = 
{\rm diam\,}f_k(\partial D_{t_k}) 
= {\rm diam\,}f_k(\partial D_{r_k \cdot \frac{t_k}{r_k}}) 
\to {\rm diam\,}f^1(\partial D_R) > 0.
$$
Now we introduce the rescalings 
$$
f^2_k:\C \to \R^3,\, f^2_k(z) = \frac{f_k(t_k z) - f_k(t_k)}{\lambda_k}.
$$
After passing to a subsequence, we have convergence of the measures 
$$
\alpha_k^2 = \mu_{f^2_k} \llcorner |A_{f^2_k}|^2 \to \alpha 
\quad \mbox{ in }C^0_c(\R^2)'. 
$$
We show that  $\alpha_k^2$ do not concentrate away from the origin. 
Assume by contradiction that $\ve_1 = \alpha(\{p\}) > 0$ for some 
$p \in \C \backslash \{0\}$. We chose $R > 0$ such that
$\alpha(\overline{D_R(p)} \backslash \{p\}) < \ve_1/2$,
and define 
$$
r_k^2 = \inf\{r > 0: \alpha^2_k(D_r(z)) \geq \frac{\ve_1}{2} 
\mbox{ for some } z \in D_R(p)\}.
$$
As $\alpha(\{p\}) = \ve_1$ we have $r_k^2 \to 0$ as $k \to \infty$. Let 
$z_k \in \overline{D_R(p)}$ be a point where the infimum is 
attained. Then $z_k \to p$, and we have 
\begin{eqnarray*}
\frac{\ve_1}{2} = \alpha^2_k(D_{r^2_k}(z_k)) \geq \alpha^2_k(D_{r^2_k}(z))
\quad \mbox{ for all }z \in \overline{D_R(p)}.
\end{eqnarray*}
We rescale the sequence again, by putting
$$
f_k^3(z) = \frac{f_k^2(z_k+r_k^2 z) - f^2_k(z_k)}{\lambda^2_k} \quad
\mbox{ where } \lambda^2_k = {\rm diam\,} f^2_k(z_k+[0,r^2_k]). 
$$
We compute
\begin{eqnarray*}
f^3_k(z) & = & \frac{1}{\lambda^2_k}\Big(
\frac{f_k(t_k(z_k+r_k^2 z)) - f_k(t_k)}{\lambda_k} 
- \frac{f_k(t_k z_k) -f_k(t_k)}{\lambda_k}\Big)\\
& = & \frac{f_k(t_k z_k + t_k r^2_k z) - f_k(t_k z_k)}{\lambda_k \lambda^2_k}.
\end{eqnarray*}
Furthermore, using that 
$f_k^2 (z_k+\varrho) = \frac{1}{\lambda_k} (f_k(t_k(z_k+\varrho))-f_k(t_k))$, 
we see that 
$$
\lambda_k \lambda^2_k = {\rm diam\,} f_k(t_k z_k + [0,t_k r^2_k]) \leq C. 
$$
The last step used the diameter estimate in \cite{S}, given the Willmore 
energy bound and the area bound for $f_k$. Now by local curvature 
estimates, see Lemma \ref{lemmahigher}, and the convergence as in
Theorem \ref{theoremconvergence}, we obtain
$$
f_k^3 \to f^3 \quad  \mbox{ locally smoothly in }\C.
$$
We conclude that $f^3:\C \to \R^3$ is a complete minimal embedding 
with total Gau{\ss} curvature at least $-4\pi$. By the known
classification of minimal immersions $f^3$ is either an Enneper 
surface or a plane. But the Enneper is not embedded, and 
the plane is also ruled out because 
$$
\frac{\ve_1}{2} = \lim_{k \to \infty} \alpha^2_k(D_{r^2_k}(z_k)) 
= \lim_{k \to \infty} \int_{D_1(0)} |A_{f^3_k}|^2\,d\mu_{f^3_k}
= \int_{D_1(0)} |A_{f^3}|^2\,d\mu_{f^3}.
$$
This contradiction shows that the sequence $f^2_k:\C \to \R^3$ 
has no curvature concentrations away from the origin. Now by the 
small energy estimates, see Lemma \ref{lemmahigher} or \cite{Riv08}, 
we conclude that
$$
f^2_k \to f^2 \quad \mbox{ locally smoothly on } \C.
$$
By the normalization ${\rm diam\,}f_k^2(\partial D) = 1$ we have 
that $f^2$ is nonconstant, hence it is a complete minimal 
embedding. Moreover we have again
$$
\int_{\C} K_{f^2}\,d\mu_{f^2} \geq - 4\pi.
$$
We also know that ${\rm diam\,}f^2(\partial D_t) \geq 1$
for all $t > 0$, and hence $f^2$ must have ends at zero and 
infinity. Altogether this implies that $f^2$ parametrizes a 
catenoid, and we have the energy identity
\begin{equation}
\label{eqenergyidentity}
24\pi = \lim_{k \to \infty} \int_{\S^2} |A_{f_k}|^2\,d\mu_{f_k}
= \int_{\S^2} |A_{f^0}|^2\,d\mu_{f^0}
+ \int_{\C} |A_{f^1}|^2\,d\mu_{f^1}
+ \int_{\C\backslash \{0\}} |A_{f^2}|^2\,d\mu_{f^2}.
\end{equation}

The convergence of the sequence $f_k$ is subsumed as follows:\vspace{2mm}

\begin{itemize}
\item[ ] $f_k(z)$ converges to $f^0:\S^2 \stackrel{\sim}{\to} \S$ 
locally smoothly on $\S^2 \backslash \{N\}$,\vspace{2mm}
\item[ ] $f^1_k(z) = f_k(r_k z)$ converges to $f^1:\hat{\C} \stackrel{\sim}{\to} \S$ 
locally smoothly on $\C$, \vspace{2mm}
\item[ ] $f^2_k(z) = \frac{1}{\lambda_k} (f_k(t_kz)-f_k(t_k))$ 
         converges to a catenoid $f^2$ smoothly on $\C \backslash \{0\}$. \vspace{4mm}
\end{itemize}
Moreover $f_k(t_k) \to f^0(0) = f^1(\infty)$, and $f^0(0)=0$
is arranged by initial translation.

\section{Asymptotics of the limit}

We finally analyse the asymptotics of the sequence $f_k$.
Our first goal is as follows.

\begin{lem} The parameters $t_k$, $\lambda_k$ and $r_k$ satisfy
$$
\lim_{k \rightarrow \infty} (\log t_k:\log\lambda_k:\log r_k) = (1:1:2).
$$
\end{lem}

\proof We again consider $f_k:\C \to \R^3$ using the chart
$\pi_{S}^{-1}:\C \to \S^2\backslash \{S\}$. The induced 
metric $(g_k)_{ij} = e^{2u_k} \delta_{ij}$ satisfies 
the Liouville equation
$$
-\Delta u_k = K_{f_k} e^{2u_k} \quad \mbox{ on } \C.
$$
Let $u_k^\ast(r) = \mint_0^{2\pi} u_k(r,\theta)\,d\theta$. It follows that 
$$
-(r (u_k^\ast)')' = 
- \mintdis_0^{2\pi} r\, \big(\partial_r^2 u_k + \frac{\partial_r u_k}{r}\big)\,d\theta
= - \mintdis_0^{2\pi} r \Delta u_k(r,\theta)\,d\theta
= \mintdis_0^{2\pi} r K_{f_k} e^{2 u_k}\,d\theta. 
$$
Thus for any $\varrho \in [\frac{t_k}{\delta},\delta]$, we can estimate
recalling our description of convergence
\begin{eqnarray*}
\sup_{\varrho \in [\frac{t_k}{\delta},\delta]} 
|\delta (u_k^\ast)'(\delta) - \varrho (u_k^\ast)'(\varrho)| 
& \leq & \int_{\frac{t_k}{\delta}}^{\delta} \mintdis_0^{2\pi} |K_{f_k}| e^{2u_k}\,r d\theta dr\\
& \leq & \frac{1}{4\pi} \int_{D_\delta \backslash D_{\frac{t_k}{\delta}}} |A_{f_k}|^2\,d\mu_{f_k}\\
& < & \varepsilon  \quad \mbox{ for all } k \geq k(\ve,\delta).
\end{eqnarray*}
This implies for $k \geq k(\ve,\delta)$ 
$$
\sup_{\varrho \in [\frac{t_k}{\delta},\delta]} \varrho |(u_k^\ast)'(\varrho)|
\leq \delta |(u_k^\ast)'(\delta)| + \ve.
$$
Now we have 
$$
(u_k^\ast)'(r) = \mintdis_0^{2\pi} \partial_r u_k(r,\theta)\,d\theta 
\to \mintdis_0^{2\pi} \partial_r u(r,\theta)\,d\theta = (u^\ast)'(r),
$$
where $u(r,\theta)$ is the conformal factor of the smooth equivalence 
$f^0:\S^2 \to \S$. We compute using the divergence theorem
$$
(u^\ast)'(\delta) = \mintdis_0^{2\pi} \partial_r u(\delta,\theta)\,d\theta = 
\mintdis_{\partial D_\delta} \langle \nabla u,\partial_r \rangle\,ds = 
\frac{1}{2\pi \delta} \int_{D_\delta} \Delta u\,dx dy \to 0 
\quad \mbox{ as } \delta \to 0.
$$
Thus $\delta |(u^\ast)'(\delta)| < \ve$ for $\delta > 0$ sufficiently 
small, and we obtain
$$
\sup_{\varrho \in [\frac{t_k}{\delta},\delta]} \varrho |(u_k^\ast)'(\varrho)| < 2\ve
\quad \mbox {for } k \geq k(\ve,\delta).
$$
Integration yields 
$$
\sup_{r \in [\frac{t_k}{\delta},\delta]} |u_k^\ast(r) - u_k^\ast(\delta)| 
\leq 2\ve\,\log \frac{\delta}{r} \quad \mbox {for } k \geq k(\ve,\delta).
$$
Using again $u^\ast_k(\delta) \to u^\ast(\delta)$ we finally get
for all $r \in [\frac{t_k}{\delta},\delta]$
$$
|u_k^\ast(r) - u^\ast(\delta)| \leq 2\ve \Big(1+\log \frac{\delta}{r}\Big) 
\quad \mbox{ if } \delta < \delta(\ve),\,k \geq k(\delta,\ve).
$$
Next, we prove that
\begin{equation}\label{osc.u}
{\rm osc}_{D_{2r}\backslash D_r} u_k \leq C \quad \mbox{ for any }
r\in [\frac{t_k}{\delta},\delta].
\end{equation}
Consider $f_{k,r}(z)=f_k(r z)$ where 
$r \in [\frac{t_k}{\delta},\delta]$. From Corollary 2.4 in 
\cite{KL12}, for any $p \in \partial D_{\frac{3}{2}}$ there exists a 
solution to the equation
$$
- \Delta \omega_{k,r} = K_{f_{k,r}} e^{2 u_{k,r}} \quad \mbox{ in }D_1(p),
$$
satisfying the estimates 
$$
\|\omega_{k,r}\|_{L^\infty(D_1(p))} + \|\nabla \omega_{k,r}\|_{L^2(D_1(p))} \leq C.
$$
Since $u_{k,r}-\omega_{k,r}$ is harmonic, we have the estimate
$$
{\rm osc}_{D_\frac{1}{2}(p)}(u_{k,r}-\omega_{k,r}) \leq 
C\,\|\nabla (u_{k,r}-\omega_{k,r})\|_{L^q(D_1(p))}
\quad \mbox{ for any } q\in (0,2).
$$
Now let $f:\S^2\rightarrow \R^3$ be a conformal immersion 
with induced metric $g = e^{2u}\,g_{\S^2}$, where $g_{\S^2}$
is the round metric normalized to $\mu_{g_{\S^2}}(\S^2) = 1$.
By Liouville we have 
$$
-\Delta_g u = K_{f}e^{2u}-4\pi.
$$
As $\|K_{f}e^{2u}-4\pi\|_{L^1}\leq 16\pi$, we have the estimate
\begin{equation}
\label{W1q}
r^{q-2}\|\nabla_{g_{\S^2}}u\|_{L^q(B_r(p))} \leq C(q) \quad
\mbox{ for any }q \in (0,2).
\end{equation}
Applying \eqref{W1q} we see that 
$$
\|\nabla u_{k,r}\|_{L^q(D_1(p))} = r^{q-2}
\|\nabla u_{k}\|_{L^q(D_r(p))} \leq C(q).
$$
By a covering argument, we get ${\rm osc}_{D_2\backslash D_1} u_{k,r} \leq C$.
But ${\rm osc}_{D_2\backslash D_1}u_{k,r} 
= {\rm osc}_{D_{2r}\backslash D_r}u_k$, thus the oscillation
bound \eqref{osc.u} is proved.\\
\\
We now come to comparing $t_k$ with $\lambda_k$. We know that
$\frac{t_k}{\lambda_k} \nabla f_k(t_kz)$ converges to $\nabla f^2$ smoothly, 
hence we have for $k$ sufficiently large
$$
2\sup_{\partial D_\frac{1}{\delta}}|\nabla f^2| \geq 
\frac{t_k}{\lambda_k}e^{u_k(t_kz)}\geq 
\frac{1}{2}\inf_{\partial D_\frac{1}{\delta}}|\nabla f^2| > 0.
$$
By \eqref{osc.u} we have $|u_k(t_kz)-u_k^*(z)| \leq C$ for $|z|=\frac{1}{\delta}$.
Hence, we get
$$
-C \leq \log \frac{t_k}{\lambda_k}+u_k^*(\frac{t_k}{\delta}) \leq C.
$$
Then
$$
-C+(1+\epsilon)\log t_k \leq \log \lambda_k \leq C+(1-\epsilon)\log t_k.
$$
Letting first $k \to \infty$ and then $\ve \to 0$, we get
$$
\lim_{k \to \infty}\frac{\log t_k}{\log \lambda_k} =1.
$$
To estimate $r_k$ we consider $\tilde{f}_k(z)=f_k(\frac{r_k}{z})$.
We compute 
$$
\tilde{f}_k(r_kz)=f_k\big(\frac{1}{z}\big) \to f_0\big(\frac{1}{z}\big),
\quad
\tilde{f}_k(z)\to f^1\big(\frac{1}{z}\big).
$$
Putting $\tilde{t}_k = r_k/t_k$ we further obtain 
$$
\frac{\tilde{f}_k(\tilde{t}_kz)}{\lambda_k}
=\frac{f_k(t_kz)}{\lambda_k}\rightarrow f^2(\frac{1}{z}).
$$
Using the arguments from above, we get
$$
\lim_{k \to \infty} \frac{\log \tilde{t}_k}{\log \lambda_k}=1,
\quad \mbox{ hence } \quad
\lim_{k \to \infty}\frac{\log r_k}{\log\lambda_k}=2.
$$
\endproof

Now we address the asymptotics of the multipliers $\Lambda_k$. For 
a conformal immersion we have in local coordinates the first 
variation formulae, see (\ref{eqwillmorevariationconformal}),
\begin{eqnarray*}
\delta \W(f)\phi & = & \int_D \langle Q[f],d\phi \rangle
\quad \mbox{ where }
Q[f] = \frac{1}{2} \big(\nabla\vec{H} - \frac{3}{2}H \,\nabla\vec{n}
+ \frac{1}{2} \vec{H} \times \nabla^\perp \vec{n}\big),\\
\delta \sigma(f)\phi & = & \int_D \langle R[f],\phi \rangle
\quad \mbox{ where }
R[f] = \frac{\sigma(f)}{\V(f)} \partial_1 f \times \partial_2 f
- \frac{3\sigma(f)}{2\mu(f)} \Delta f.
\end{eqnarray*}
The scaling is $Q[\lambda f] = \lambda^{-1} Q[f]$ and also
$R[\lambda f] = \lambda^{-1} R[f]$. The Euler Lagrange equation
(\ref{eqmultiplier}), scaled to $\mu(f) = 1$, has the form
\begin{equation}
\label{eqeulerlagrangeconform2}
\div Q[f] = \frac{3\Lambda}{2}
\big(4 \sqrt{\pi}\, \partial_1 f \times \partial_2 f
- \sigma \Delta f\big):=\frac{3\Lambda}{2}S[f].
\end{equation}
For simpilicity we assume $0\notin f^2(\C\setminus\{0\})$.
Put $I(y) = \frac{y}{|y|^2}$ and $I_k(y)= I(\frac{y}{\lambda_k}) = \lambda_k I(y)$. 
Note that $I_k \circ I_k = {\rm id}$. From the previous section, we 
compute using the assumption $f_k(t_k) = 0$ 
$$
I_k(f_k(t_k z)) = I\Big(\frac{f_k(t_k z))}{\lambda_k}\Big) \to I(f^2(z)) \quad 
\mbox{ for all } z \in \R^2.
$$
We need to compute the equation satisfied by the maps $F_k = I_k \circ f_k$. 
We compute
\begin{eqnarray*}
\delta \W(F_k,\phi) & = & \frac{d}{dt} \W(F_k+t\phi)|_{t=0}\\
& = & \frac{d}{dt}\W(I_k \circ (F_k+t\phi))|_{t=0}\\
& = & \delta \W(I_k(F_k),D_{F_k}I_k(\phi))\\
& = & \delta \W(f_k,D_{F_k}I_k(\phi))\\
&=&\Lambda_k\delta \sigma(f_k,D_{F_k}I_k(\phi))
\end{eqnarray*}
Then we have
\begin{eqnarray*}
\int_D\lan Q[F_k],d\phi\ran dx&=&
\Lambda_k\sigma(f_k)\int_D\lan R[f_k],D_{F_k}I_k(\phi)\ran dx\\
&=&\frac{3}{2}\Lambda_k\int_D\lan S[f_k],D_{F_k}I_k(\phi)\ran dx.
\end{eqnarray*}
Since $I_k(F_k+t\phi)=I_k(F_k+t\phi)=\frac{\lambda_k(F_k+t\phi)}{|F_k+t\phi|^2}$, 
we get 
$$
D_{F_k}I_k(\phi) 
= \frac{d}{dt} I_k \circ (F_k + t\phi)|_{t=0} 
= \lambda_k\left(\frac{\phi}{|F_k|^2} - 2\frac{F_k}{|F_k|^4}F_k\cdot \phi\right). 
$$
Then we have
\begin{eqnarray*}
\int_D\lan Q[F_k],d\phi\ran dx=\frac{3}{2}\lambda_k\Lambda_k\int_D\left(
\frac{1}{|F_k|^2}\left\lan S[f_k],\phi\right\ran
-2\frac{1}{|F_k|^4}\lan S[f_k],F_k\ran \lan F_k,\phi\ran\right) dx.
\end{eqnarray*}
Recalling that $F_k=\lambda_k\frac{f_k}{|f_k|^2}$, we get
$$
\div Q[F_k]=\frac{3}{2}\frac{\Lambda_k}{\lambda_k}
\left(|f_k|^2S[f_k]-2\lan S[f_k],f_k\ran f_k\right).
$$
Since $S^2\setminus D_{t_k}$ is also conformal to $D$, we compute 
\begin{eqnarray*}
-\int_{\partial D_{t_k}}\left\lan Q[F_k],\partial_r\right\ran &=&\int_{S^2\backslash D_{t_k}}{\div} Q[F_k]\\
&=&\frac{3}{2}\frac{\Lambda_k}{\lambda_k}\int_{S^2\backslash D_{t_k}}
\left(|f_k|^2S[f_k]-2\lan S[f_k],f_k\ran f_k\right)\\
&=&\frac{3}{2}\frac{\Lambda_k}{\lambda_k}\int_{S^2\backslash D_{t_k}}
e^{-2u_k}\left(|f_k|^2S[f_k]-2\lan S[f_k],f_k\ran f_k\right)d\mu_{f_k}.
\end{eqnarray*}
Obviously,
$$
\lim_{k\rightarrow+\infty}
\int_{\partial D_{t_k}}\lan Q[F_k],\partial_r\ran ds
=\lim_{k\rightarrow+\infty}\int_{\partial D_1}\lan Q[I(\frac{f_k(t_kz)}{\lambda_k})],
\partial_r\ran d\theta=\int_{\partial D_1}
\lan Q[I(f^2)],\partial_r\ran d\theta.
$$
We have
$$
\int_{S^2\backslash D_{t_k}}e^{-2u_k}\left(|f_k|^2S[f_k]-2\lan S[f_k],f_k\ran f_k\right)d\mu_{f_k}=\int_{S^2\backslash D_\delta}\cdots+\int_{D_\delta\backslash D_{t_k}}
\cdots.$$
Since
$$\left|e^{-2u_k}(|f_k|^2S[f_k]-2\lan S[f_k],f_k\ran f_k\right|<C(\sigma(f_k)|H_{f_k}|+1),$$
we have
\begin{eqnarray*}
&& \left|\int_{D_\delta\backslash D_{t_k}}e^{-2u_k}\left(|f_k|^2S[f_k]-2\lan S[f_k],f_k\ran f_k\right)d\mu_{f_k}\right|\\
&& \leq C\left(\int_{D_\delta\backslash D_{t_k}}\sigma(f_k)|H_{f_k}|d\mu_{f_k}+\mu_{f_k}(D_\delta\backslash D_{t_k})\right)\\
&& \leq C\left(\sigma(f_k)\sqrt{W(f_k)\mu_{f_k} (D_\delta\backslash D_{\frac{r_k}{\delta}})}
+ \mu_{f_k}(D_\delta\backslash D_{\frac{r_k}{\delta}})\right).
\end{eqnarray*}
By \eqref{eqchenli1},
$$
\lim_{\delta\rightarrow 0}\lim_{k\rightarrow+\infty}
\int_{D_{\delta}\backslash D_{t_k}}e^{-2u_k}\left(|f_k|^2S[f_k]-2\lan S[f_k],f_k\ran f_k\right)d\mu_{f_k}=0.
$$
By a direct calculation, on $S^2\backslash D_\delta$, we have
$$
\lim_{k\rightarrow+\infty}e^{-2u_k}\left(|f_k|^2S[f_k]-2\lan S[f_k],f_k\ran f_k\right)=
4\sqrt{\pi}(n_{f_0}|f_0|^2-2f_0\cdot n_{f_0} f_0).
$$
Then
$$
\lim_{\delta\rightarrow 0}\lim_{k\rightarrow+\infty}
\int_{S^2\backslash D_{\delta}}e^{-2u_k}(|f_k|^2S[f_k]-2\lan S[f_k],f_k\ran f_k)d\mu_{f_k}
=4\sqrt{\pi}\int_{S^2}\left(n_{f_0}|f_0|^2-2f_0\cdot n_{f_0} f_0\right)d\mu_{f_0}.
$$
Let $y_0$ be the center of $f^1$,  which is nonzero because $0 \in f^1(S^2)$,
and $R=\frac{1}{\sqrt{8\pi}}$.
Then 
$$
\begin{array}{lll}
\dint_{S^2}(n_{f_0}|f_0|^2-2f_0\cdot n_{f_0} f_0)d\mu_{f_0}
&=&\dint_{S^2}\left(n_{f_0}|Rn_{f_0}+y_0|^2-2\lan Rn_{f_0}+y_0,n_{f_0}\ran(Rn_{f_0}+y_0)\right)d\mu_{f_0} \\[\mv]
&=&\dint_{S^2}(-R^2n_{f_0}+|y_0|^2n_{f_0}-2Ry_0-2\lan y_0,n\ran y_0)d\mu_{f_0}\\[\mv]
&=&-2Ry_0\int_{S^2}d\mu_{f_0}=-Ry_0
\end{array}$$
where we used that 
$$
\int_{S^2}n_{f_0}=0 \quad \mbox{ and } \quad \int_{S^2}n_{f_0}\cdot y_0=0.
$$
Hence finally
\begin{equation}
\int_{\partial D_1}
\lan Q[I_0(f^2)],\partial_r\ran=\frac{3}{2}\sqrt{2}y_0\lim_{k\rightarrow+\infty}\frac{\Lambda_k}{\lambda_k}.
\end{equation}\\
It is well-known that $\int_{\partial D_1}
\lan Q[I(f^2)],\partial_r\ran$ is a nonzero vector parallel to the
symmetry axis of $f^2$ (cf \cite{KS12} or \cite{BR13}).

\section{Appendix1}

In this appendix, we use results from \cite{KS01} to derive estimates 
for the equation
\begin{equation}\label{Willmore}
\Delta_g H + |A^\circ|^2 H=aH+b, \quad \mbox{ for $a,b \in \R$ constant.}
\end{equation}
We start with higher order estimates when the curvature is not concentrated.

\begin{lem}\label{lemmahigher}
Let $f:\Sigma\rightarrow\R^3$ be a properly immersed surface which
satisfies \eqref{Willmore} for $|a|+|b|<\lambda$.
Let $\Sigma_1=f^{-1}(B_1(x_0)) \subset \!\subset\Sigma$,
and assume that $\int_{\Sigma_1}|A|^2 < \epsilon_0$. Then
$$
\|\nabla^mA\|_{L^\infty(\Sigma_\frac{1}{2})} < C=C(m,\lambda,\Lambda).
$$
\end{lem}

\proof Applying Theorem 2.10 in \cite{KS01} yields 
\begin{eqnarray*}
\|A\|_{L^\infty(\Sigma_{\frac{\varrho}{2}}(x_0))} & \leq &
\frac{C}{\varrho}\|A\|_{L^2(\Sigma_\varrho(x_0))} +
C \varrho \big(|a|\,\|H\|_{L^2(\Sigma_\varrho(x_0))} 
+ |b| \mu(\Sigma_\varrho(x_0))^{\frac{1}{2}}\big)\\
& \leq & \frac{C}{\varrho}\|A\|_{L^2(\Sigma_\varrho(x_0))}
+ C\, \big(|a| \varrho + |b| \varrho^2),
\end{eqnarray*}
where we used the quadratic area growth \cite{S}. In particular 
\begin{equation}
\label{eqhigherfundinfty}
\|A\|_{L^\infty(\Sigma_\frac{3}{4})} \leq C.
\end{equation}
Now each $x \in f(\Sigma)\cap B_\frac{1}{2}(x_0)$ has finite preimage
$\{y^1,\ldots,y^m\}$. By Langer's theorem \cite{La85}, for each 
$\alpha \leq 1$ there is a constant $r > 0$ such that $f$ is an 
$(r,\alpha)$-immersion on $\Sigma_\frac{3}{4}$. This means that 
on suitable neighborhoods $U^i$ of $y^i$, the $f|_{U_i}$ are
graphs over $P^i \cap B_r(x)$, where $P^i = {\rm im}\, Df(y^i)$
are the tangent planes, with graph functions $u^i$ satisfying 
$$
\|\nabla u^i\|_{C^{0}(B_r(x)\cap P^i)} \leq \alpha. 
$$
In fact Langer shows in Lemma 2.1 in \cite{La85} that 
$$
\|u^i\|_{W^{2,p}(B_r(x)\cap P^i)}\leq C(p) \quad \mbox{ for all }p < \infty.
$$
It follows that the metric $g_{\lambda \mu}$ satisfies in graph coordinates
$$
\|g_{\lambda \mu}\|_{C^{0,\beta}} \leq C, \quad 
|g_{\lambda \mu}-\delta_{\lambda\mu}| \leq C\alpha. 
$$
For $f(x,y) = (x,y,u(x,y))$ we now have the equations
\begin{eqnarray}
\label{eqhighermetric}
g_{\lambda \mu} & = & \delta_{\lambda \mu} + \partial_\lambda u\, \partial_\mu u ,\\
\label{eqhigherwillmore}
g^{\lambda \mu} \partial^2_{\lambda \mu} H 
+ \frac{\partial_\lambda (\sqrt{\det g}\, g^{\lambda \mu})}{\sqrt{\det g}}\,\partial_\mu H
& = & -|A^\circ|^2 H + a H + b,\\
\label{eqhighermean}
\Big(\delta_{\lambda \mu} 
- \frac{\partial_\lambda u\, \partial_\mu u}{1+|Du|^2}\Big)\,
\partial^2_{\lambda \mu} u & = &  H\,\sqrt{1+|Du|^2},\\
\label{eqhigherfund}
A_{\lambda \mu}^3 & = & 
\Big(1-g^{\gamma \nu} \partial_\gamma u\,\partial_\nu u\Big)\,\partial^2_{\lambda \mu} u,
\end{eqnarray}
From (\ref{eqhigherfundinfty}) and (\ref{eqhigherfund}) we know that 
$\nabla^2 u$ is bounded, and so is $\nabla g_{\lambda \mu}$ by (\ref{eqhighermetric}).
Therefore we can apply $L^p$-theory to (\ref{eqhigherwillmore}), see 
Theorem 9.11 in \cite{GT}, to show $H \in W^{2,p}$ locally 
for any $p < \infty$. The right hand side of (\ref{eqhighermean}) then
belongs to $C^{1,\alpha}$ for any $\alpha < 1$, and the coefficients 
of the equation are also in $C^{1,\alpha}$. Schauder estimates give 
$u \in C^{3,\alpha}$, which in turn yields $g \in C^{2,\alpha}$
and $A \in C^{1,\alpha}$. Returning to (\ref{eqhigherwillmore}) 
improves to $H \in C^{3,\alpha}$. The lemma follows by iterating 
Schauder estimates. \endproof


The next lemma shows a relation between a conformal map and an extrinsic
estimate.

\begin{lem}
Let $f\in C^\infty(D,\R^3)$, $f(0) = 0$, with metric $g=e^{2u}g_{euc}$.
Assume
\begin{equation}
\label{eqappendix1}
\int_{D}|A|^2 \leq 8\pi-\tau \quad \mbox{ and } \quad |u| \leq \beta.
\end{equation}
Let $C_\delta(f)$ be the component of $f^{-1}(B_\delta(0))$ 
containing the origin. Then for any $r > 0$ there exists
$\delta=\delta(\tau,\beta,r) > 0$ such that
$$
C_\delta(f)\subset D_r.
$$
Likewise, for any $\delta > 0$ there exists $\rho=\rho(\tau,\beta,\delta) > 0$
such that
$$
D_\rho \subset C_\delta(f). 
$$
\end{lem}

\begin{proof}
If the first statement fails then we can find $\delta_k \to 0$
and $f_k \in C^\infty(D,\R^3)$ with (\ref{eqappendix1}), where
$f_k(0) = 0$ and $g_{f_k}=e^{2u_k}g_{euc}$, such that there exist
$$
z_k\in C_{\delta_k}(f_k) \quad \mbox{ with } |z_k| \geq r.  
$$
By results in \cite{Hel02} and \cite{KL12}, $f_k$
converges to some $f \in W^{2,2}_{conf}(D,\R^3)$ locally 
weakly in $W^{2,2}(D)$ and strongly in $C^0(D,\R^3)$.
Now there are curves $\gamma_k:[0,1] \to D$ with $\gamma_k(0) = 0$,
$\gamma_k(1) = z_k$ and $|f_k(\gamma_k(t))| < \delta_k$ 
for all $t \in [0,1]$. It follows that 
$$
\min_{z \in \partial D_\varrho} |f(z)| = 0 \quad \mbox{ for all }
\varrho \in (0,r].
$$
By the local expansion for $W^{2,2}$ immersions, we conclude 
that $f$ is constant. But $|Df|^2 \geq \frac{1}{2} e^{-2\beta} > 0$,
a contradiction.\\ 
\\
If the second claim was not true, then there exist points
$z_k \notin C_\delta(f_k)$ such that $z_k \to 0$, where
$f_k \in C^\infty(D,\R^n)$ satisfies $f_k(0) = 0$ and (\ref{eqappendix1}).
Now we have either $|f(z_k)|\geq \delta$, or $z_k$ is in
another component of $f^{-1}(B_\delta(0))$. In the second 
we can find $\lambda_k \in (0,1)$ such that 
$|f(\lambda_k z_k)| \geq \delta$. Thus, for both cases, we have
$\lambda_k \in (0,1]$ such that $|f(\lambda_k z_k)|\geq \delta$.\\
\\ 
As above the sequence converges to a $W^{2,2}$ conformal immersion 
$f:D \to \R^3$, locally uniformly in $D$. But then
$$
0 = \lim_{k \to \infty} |f_k(0)| 
= \lim_{k \to \infty} |f_k(\lambda_k z_k)| 
\geq \delta,
$$
which is again a contradiction. \end{proof}

Applying the above lemma, we obtain the following:

\begin{pro}
Let $f:D \to \R^3$ be a conformal immersion with 
metric $g_{f}=e^{2u}g_{euc}$, which satisfies
\eqref{Willmore} with $|a|+|b| \leq \alpha$.
We assume $|u| \leq \beta$. There exists a constant
$\epsilon_0 > 0$, such that if $\int_D |A|^2\,d\mu_f < \epsilon_0$,
then for any $r < 1$ and $m > 0$
and $m>0$
$$
\|\nabla^m f\|_{L^\infty(D_r)} \leq C(m,r,\alpha,\beta).
$$
\end{pro}

The proof uses a priori estimates similar to Lemma \ref{lemmahigher}
and is omitted. 

\section{Appendix2}
\begin{thm}\label{theoremconvergence}
Let
$f_k\in W^{2,2}(D,\R^n)$  satisfy
\begin{itemize}
\item[{\rm 1)}] $\int_{D}|A_{f_k}|^2d\mu_{f_k}<4\pi-\tau$.
\item[{\rm 2)}] $f_k$ can be extended to a closed 
immersed surface $f_k:\Sigma_k \to \R^n$ with
$$\int_{\Sigma_k}|A_{f_k}|^2d\mu_{f_k}<\Lambda.$$
\end{itemize}
Take a curve $\gamma:[0,1]\rightarrow D$,
and set $\lambda_k=diam\, f_k(\gamma)$.
Then
we can find  a subsequence of $\frac{f_k
-f_k(\gamma(0))}{\lambda_k}$ which
converges  weakly in
$W^{2,2}_{loc}(D)$ to an
$f_0\in W_{conf,loc}^{2,2}(D,\R^n)$.
Furthermore, we can find an inverse $I=\frac{y-y_0}{|y-y_0|^2}$
with $y_0\notin f_0(D)$ such that
$$\int_\Sigma(1+|A_{I(f_0)}|^2)d\mu_{I(f_0)}<+\infty.$$
\end{thm}

\proof
Put $f_k'=\frac{f_k-f_k(\gamma(0))}{\lambda_k}$,
$\Sigma_k'=\frac{\Sigma_k-f_k(\gamma(0))}{\lambda_k}$.
We have two cases:\vspace{0.7ex}

\noindent Case 1: $diam(f_k')<C$. By
inequality (1.3) in \cite{S} with $\rho=\infty$,
$\frac{\Sigma_k'\cap B_\sigma(\gamma(0))}{\sigma^2}\leq C$ for any $\sigma>0$.
Hence we get $\mu(f_k')<C$ by taking
$\sigma=diam(f_k')$. Then by Helein's convergence theorem
\cite{Hel02,KL12},
$f_k'$ converges weakly
in $W^{2,2}_{loc}(D)$. Since
$diam\, f_k'
(\gamma)=1$, the weak limit is not trivial.\\

\noindent Case 2: $diam(f_k')\rightarrow +\infty$. We take a point
$y_0\in\R^n$ and a constant $\delta>0$, s.t.
$$B_\delta(y_0)\cap \Sigma_k'=\emptyset,\s \forall k.$$
Let $I=\frac{y-y_0}{|y-y_0|^2}$, and
$$f_k''=I(f_k'),\s \Sigma_k''=I(\Sigma_k').$$
By conformal invariance of Willmore functional we have 
$$\int_{\Sigma_k''}|A_{\Sigma_k''}|^2d\mu_{\Sigma_k''}
=\int_{\Sigma_k}|A_{\Sigma_k}|^2d\mu_{\Sigma_k}<\Lambda.$$
Since $\Sigma_k''\subset B_\frac{1}{\delta}(0)$, also by (1.3) in \cite{S},
we get $\mu(f_k'')<C$. Let
$$\mathcal{S}(\{f_k''\}):=
\{p\in D: \lim\limits_{r\rightarrow 0}\varliminf\limits_{k\rightarrow+\infty}
\int_{D_(p)}|A_{f_k''}|^2d\mu_{f_k''}\geq 4\pi \}.$$
Then
$f_k''$ converges weakly in $W^{2,2}_{loc}(D\backslash 
\mathcal{S}(f_k''))$.

Next, we prove that $f_k''$ does not converge to a point by assumption.
If $f_k''$ converges to a point in
$W^{2,2}_{loc}(D\backslash \mathcal{S}(f_k''))$,
then the limit must be 0,  for $diam\,(f_k')$
converges to $+\infty$.
By the
definition of $f_k''$, we can find a $\delta_0>0$,
such that $f_k''(\gamma)\cap
B_{\delta_0}(0)=\emptyset$. Thus for any $p\in \gamma([0,1])
\backslash \mathcal{S}(f_k'')$, $f_k''$ will not converge to $0$. A contradiction.

Then we only need to prove that $f_k'$ converges weakly in
$W^{2,2}_{loc}(D,\R^n)$.
Let $f_0''$ be the limit of $f_k''$ which is a branched immersion of $D$ in $\R^n$.
Let $\mathcal{S}^*=f_0^{''-1}(\{0\})$,
which is isolate.

First, we prove that for any $\Omega\subset\subset D\backslash 
(\mathcal{S}^*\cup\mathcal{S}(\{f_k''\}))$, $f_k'$
converges weakly in $W^{2,2}(D,\R^n)$:
Since $f_0''$ is continuous
on $\bar{\Omega}$, we may assume
$dist(0,f_0''(\Omega))>\delta>0$. Then $dist(0,f_k''(\Omega))>\frac{\delta}{2}$
when $k$ is sufficiently large. Noting that $f_k'
=\frac{f_k''}{|f_k''|^2}+y_0$, we get that $f_k'$ converges weakly in
$W^{2,2}(\Omega,\R^n)$.

Next, we prove that for each
$p\in \mathcal{S}^*\cup\mathcal{S}(\{f_k''\})$, $f_k'$ also converges in
a neighborhood of $p$.

Let $g_{f_k'}=e^{2u_k'}g_{euc}$.
Since $f_k'\in W^{2,2}_{conf}
(D_{4r}(p))$ with $\int_{D_{4r}(p)}|A_{f_k'}|^2d\mu_{f_k'}<4\pi-\tau$ when $r$ is
sufficiently small and $k$ sufficiently large,
by the arguments in \cite{KL12},
we can find a $v_k$ solving the equation
$$-\Delta v_k=K_{f_k'}e^{2u_k'},\s z\in D_r\s and\s \|v_k\|_{L^\infty(D_r(p))}<C.$$
Since $f_k'$ converges to a conformal
immersion in $D_{4r}\backslash D_{\frac{1}{4}r}(p)$,
we may assume that
$$\|u_k'\|_{L^\infty(D_{2r}\backslash 
D_r(p))}<C.$$
 Then
$u_k'-v_k$ is a harmonic function with
$\|u_k'-v_k\|_{L^\infty(\partial D_{2r}(p))}<C$,
then we get $\|u_k'(z)-v_k(z)\|_{L^\infty(D_{2r}(p))}<C$
from the Maximum Principle. Thus, $\|u_k'\|_{L^\infty(D_{2r}(p))}<C$,
which implies $\|\nabla f_k'\|_{L^\infty(D_{2r})}<C$.
By the equation $\Delta f_k'=e^{2u_k'}H_{f_k'}$, and
the fact that
$$\|e^{2u_k'}H_{f_k'}\|_{L^2
(D_{2r})}^2<
e^{2\|u_k'\|_{L^\infty}}\int_{D_{2r}}|H_{f_k'}|^2d\mu_{{f_k'}},$$
we get $\|\nabla{f_k'}\|_{W^{1,2}(D_{r})}<C$.
We complete the proof.
\endproof

\end{document}